\documentclass[12pt]{article}
\usepackage[centertags]{amsmath}
\usepackage{amsfonts}
\usepackage{amssymb}
\usepackage{latexsym}
\usepackage{amsthm}
\usepackage{newlfont}
\usepackage{graphicx}
\usepackage{listings}
\usepackage{booktabs}
\usepackage{abstract}
\usepackage{xcolor}

\usepackage{cancel}
\usepackage{ulem}

\date{}

\newlength{\defbaselineskip}
\newcommand{\setlinespacing}[1]%
           {\setlength{\baselineskip}{#1 \defbaselineskip}}

\newcommand{\N}{{\mathbb{N}}}

\newcommand{\actaqed}{\hfill $\actabox$}
{\medskip\noindent \textit{Proof of #1. }}%
{\actaqed \medskip}

\def\C{{\mathcal C}}

\def \cM{\mathcal M}
\def\R{{\mathbb R}}

\def\bbC{\mathbb C}
\def \<{\langle}
\def\>{\rangle}

\def \ep{\epsilon}

\def \de{\delta}

\def\la{\lambda}
\def\ro{\varrho}

\def \sp{\operatorname{span}}

\def\bv{\mathbf v}
\def\bw{\mathbf w}

\def\bF{\mathbf F}

\newtheorem{Theorem}{Theorem}[section]
\newtheorem{Lemma}{Lemma}[section]

\newtheorem{Definition}{Definition}[section]
\newtheorem{Proposition}{Proposition}[section]
\newtheorem{Remark}{Remark}[section]

\newtheorem{Corollary}{Corollary}[section]

\numberwithin{equation}{section}

\newcommand{\be}{\begin{equation}}
\newcommand{\ee}{\end{equation}}

\begin{document}

\title{On sampling discretization in $L_2$}

\author{I. Limonova\thanks{Lomonosov Moscow State University, Moscow Center for Fundamental and Applied Mathematics,} \,  and V. Temlyakov\thanks{University of South Carolina, Steklov Institute of Mathematics, Lomonosov Moscow State University, and Moscow Center for Fundamental and Applied Mathematics.} } \maketitle

\begin{abstract}
	{We prove a sampling discretization theorem for the square norm of functions from a finite dimensional subspace satisfying Nikol'skii's inequality with an upper bound on the number of sampling points of the order of the dimension of the subspace.}
\end{abstract}


\section{Introduction}
\label{Int}

Let $\Omega$ be a   nonempty subset of $\R^d$ with  the probability measure $\mu$. By $L_q$, $1\le q< \infty$, norm we understand
$$
\|f\|_q:=\|f\|_{L_q(\Omega,\mu)} := \left(\int_\Omega |f|^qd\mu\right)^{1/q}.
$$
By discretization of the $L_q$-norm we understand a replacement of the measure $\mu$ by
a discrete measure $\mu_m$ with support on a set $\xi =\{\xi^j\}_{j=1}^m \subset \Omega$. This means that integration with respect to  measure $\mu$ is replaced by an appropriate cubature formula. Thus, integration is replaced by evaluation of a function $f$ at a
finite set $\xi$ of points. This method of discretization is called {\it sampling discretization}. Discretization is
an important step in making a continuous problem computationally feasible. The reader can find a corresponding discussion in a recent survey \cite{DPTT}. The first results in sampling discretization were obtained by Marcinkiewicz and
by Marcinkiewicz-Zygmund (see \cite{Z}) for discretization of the $L_q$-norms of the univariate trigonometric polynomials in 1930s. Therefore,  sampling discretization results are sometimes referred to as {\it Marcinkiewicz-type theorems} (see \cite{VT158}, \cite{VT159}, \cite{DPTT}). Recently,  substantial progress in sampling discretization has been made in \cite{VT158}, \cite{VT159}, \cite{VT168}, \cite{DPTT},  \cite{DPSTT1}, \cite{DPSTT2}, \cite{Kos}. 

Let us   comment  on the values of functions from $L_q$. We are interested in discretization of the $L_q$-norms, $1\le q\le \infty$, of elements of finite dimensional subspaces. By a function $f\in L_q(\Omega,\mu)$ we
understand a specific function (not an equivalency class), which is defined almost everywhere with respect to $\mu$ on $\Omega$. In other words, for $f\in L_q(\Omega,\mu)$ there exists a set $E(f)\subset \Omega$ such that $\mu(E(f))=0$ and $f(x)$ is defined for all $x\in \Omega \setminus E(f)$.
We say that a subspace $X_N\subset L_q(\Omega,\mu)$ is an $N$-dimensional subspace if there are $N$ linearly independent functions $u_i\in X_N$, $i=1,\dots,N$, such that
$X_N=\sp(u_1,\dots,u_N)$. In this case, for the subspace $X_N$ there exists a set $E(X_N)\subset \Omega$ such that $\mu(E(X_N))=0$ and each $f\in X_N$ is defined for all $x\in \Omega \setminus E(X_N)$. It will be convenient for us to assume that each $f\in X_N$ is defined for all $x\in \Omega$.

In this paper we present results on sampling discretization in the case $q=2$.  We consider two settings: (I) discretization with equal weights and (II) weighted discretization.  
In Section \ref{A} we prove the main technical results -- Lemma \ref{Lim} and its Corollary \ref{weighted} -- that are used in the proofs of Theorems \ref{IT4} and \ref{ITw}. We also discuss in detail known results related to the main lemma -- Lemma \ref{Lim}. In Section \ref{B} we prove and discuss Theorems \ref{IT4} and \ref{ITw}. 

We now proceed to a detailed discussion of our new results and related known results. First, in Subsection \ref{Ia} we formulate the main results of the paper and   comment 
on their novelty and impact. Second, in Subsection \ref{Ib} we present a brief history of discretization with equal weights, which is directly related to our Theorem \ref{IT4}. Finally, in Subsection \ref{Ic} we give a historical comment on weighted discretization.

\subsection{Main results}\label{Ia}

{\bf I. Equal weights.} The following condition is the key to the existence of good discretization with equal weights.

{\bf Condition E.} We say that an orthonormal system $\{u_i(x)\}_{i=1}^N$ defined on $\Omega$ satisfies Condition E with a constant $t>0$   if for all $x\in \Omega$
$$
 \sum_{i=1}^N |u_i(x)|^2 \le Nt^2.
$$
 Note that integration of the above inequality over $x\in\Omega$ gives  $t\geq 1$. 

The following Theorem \ref{IT4} solves (in the sense of order) the problem 
of discretization with equal weights for $N$-dimensional subspaces of $L_2(\Omega,\mu)$ satisfying 
Condition E. 

 \begin{Theorem}\label{IT4} Let  $\Omega\subset \R^d$ be a   nonempty set with  the probability measure $\mu$. Assume that $\{u_i(x)\}_{i=1}^N$ is a real (or complex) orthonormal system in $L_2(\Omega,\mu)$ satisfying Condition E. 
Then there is an absolute  constant $C_1$ such that there exists a set $\{\xi^j\}_{j=1}^m\subset \Omega$ of $m \le C_1 t^2 N$ points with the property:
 for any $f=\sum_{i=1}^N c_iu_i$  we have  
\begin{equation*}
C_2 \|f\|_2^2 \le \frac{1}{m}\sum_{j=1}^m |f(\xi^j)|^2 \le C_3 t^2\|f\|_2^2, 
\end{equation*}
where $C_2$ and $C_3$ are absolute positive constants. 
\end{Theorem}

It is known that Condition E is equivalent to the fact that the subspace $X_N:= \sp(u_1,\dots,u_N)$ 
satisfies the Nikol'skii inequality for the pair $(2,\infty)$ (see Section \ref{B} for a detailed discussion). In Section \ref{B} we reformulate Theorem \ref{IT4}
in terms of Nikol'skii's inequality (see Theorem \ref{BT4} and Remark \ref{BR2}) and prove it.

{\bf II. General weights.}  Theorem \ref{ITw} solves (in the sense of order) the problem 
of weighted discretization for arbitrary $N$-dimensional subspaces of $L_2(\Omega,\mu)$. 

\begin{Theorem}\label{ITw} If $X_N$ is an $N$-dimensional subspace of the complex $L_2(\Omega,\mu)$, then there exist three absolute positive constants $C_1'$, $c_0'$, $C_0'$,  a set of $m\leq   C_1'N$ points $\xi^1,\ldots, \xi^m\in\Omega$, and a set of nonnegative  weights $\lambda_j$, $j=1,\ldots, m$,  such that
\[ c_0'\|f\|_2^2\leq  \sum_{j=1}^m \lambda_j |f(\xi^j)|^2 \leq  C_0' \|f\|_2^2,\  \ \forall f\in X_N.\]
\end{Theorem}

\begin{Remark}\label{IRw} {We formulate Theorem \ref{ITw} in terms of the class of nonnegative  weights $\lambda_j$, $j=1,\ldots, m$. This class of weights is a standard class in numerical integration (cubature formulas) and in discretization. Clearly, the terms with $\lambda_j=0$ can be dropped and then we come to the class of positive weights, provided at least one of the weights is positive}. 
\end{Remark}

Theorem \ref{ITw} is proved in Section \ref{B}. For the reader's convenience it is formulated there again as 
Theorem \ref{BT3}.

\bigskip

{\bf Novelty and impact.} Theorems \ref{IT4} and \ref{ITw} solve (in the sense of order) two sampling discretization problems at a reasonable level of generality. In the case of the Marcinkiewicz-type discretization with equal weights we only impose Condition E, which is a standard condition in this case (see
Subsection \ref{Ib} for details). Theorem \ref{ITw} provides a discretization result for any subspace of $L_2$, which is important for applications. A preprint version of this paper has been published in \cite{LT}.
It has made an immediate impact on research of the optimal sampling recovery. Theorem \ref{ITw} was used in \cite{VT183} for proving an important inequality for the optimal sampling recovery. For the reader's 
convenience we formulate it here. Recall the setting 
 of the optimal recovery. For a fixed $m$ and a set of points  $\xi:=\{\xi^j\}_{j=1}^m\subset \Omega$, let $\Phi_\xi $ be a linear operator from $\bbC^m$ into $L_p(\Omega,\mu)$.
Denote for a class $\bF$ (usually, centrally symmetric and compact subset of $L_p(\Omega,\mu)$)
$$
\varrho_m(\bF,L_p) := \inf_{\text{linear}\, \Phi_\xi; \,\xi} \sup_{f\in \bF} \|f-\Phi_\xi(f(\xi^1),\dots,f(\xi^m))\|_p.
$$
The following statement was proved in \cite{VT183}. There exist two positive absolute constants $b$ and $B$ such that for any   compact subset $\Omega$  of $\R^d$, any probability measure $\mu$ on it, and any compact subset $\bF$ of $\C(\Omega)$ we have
\be\label{osrin}
\ro_{bn}(\bF,L_2(\Omega,\mu)) \le Bd_n(\bF,L_\infty).
\ee
Here, $d_n(\bF,L_\infty)$ is the Kolmogorov width of $\bF$ in the uniform norm. In turn, inequality (\ref{osrin}) was used in \cite{TU1} to prove new bounds for the optimal sampling recovery of functions 
with small mixed smoothness. 

The proof of Theorem \ref{IT4} is based on Lemma \ref{Lim}. A version of this lemma (see Remark \ref{LimR} below) was used in \cite{NSU} for a breakthrough result on the sampling recovery. 

Note that our proofs of Theorems \ref{IT4} and \ref{ITw} are not technically involved because they are based on deep known results.

 \subsection{Historical comments on discretization with equal weights}\label{Ib}
 
 We begin with the formulation of Rudelson's result from \cite{Rud}. In the paper \cite{Rud} it is formulated in terms of submatrices of an orthogonal matrix. We reformulate it in our setting. 
Note that Theorem \ref{Rud} can be derived from the original result of Rudelson in the same way as we derive Theorem \ref{BT1} from Lemma \ref{Lim} (see Section \ref{B} below). 

\begin{Theorem}[\cite{Rud}]\label{Rud}  Let $\Omega_M=\{x^j\}_{j=1}^M$ be a discrete set with the probability measure $\mu_M(x^j)=1/M$, $j=1,\dots,M$. Assume that 
  a real orthonormal system $\{u_i(x)\}_{i=1}^N$   satisfies   Condition E on $\Omega_M$. 
 Then for every $\ep>0$ there exists a set $J\subset \{1,\dots,M\}$ of indices with cardinality 
\be\label{5.1a}
m:=|J| \le C\frac{t^2}{\ep^2}N\log\frac{Nt^2}{\ep^2}
\ee
such that for any $f=\sum_{i=1}^N c_iu_i$ we have
$$
(1-\ep)^2\|f\|_2^2 \le \frac{1}{m} \sum_{j\in J} f(x^j)^2 \le (1+\ep)^2\|f\|_2^2.
$$
\end{Theorem}

In \cite{VT159}  it was demonstrated how the Bernstein-type concentration inequalities for random matrices can be used to prove an analog of Theorem \ref{Rud}  for a general $\Omega$. The proof in \cite{VT159} is based on a different idea than the Rudelson's proof. 
Here is the corresponding result.  

\begin{Theorem}[{\cite[Theorem 6.6]{VT159}}]\label{IT2} Let $\{u_i(x)\}_{i=1}^N$ be a real orthonormal in $L_2(\Omega,\mu)$ system satisfying Condition E. Then for every $\ep>0$ there exists a set $\{\xi^j\}_{j=1}^m \subset \Omega$ with
$$
m  \le C\frac{t^2}{\ep^2}N\log N
$$
such that for any $f=\sum_{i=1}^N c_iu_i$ we have
$$
(1-\ep)\|f\|_2^2 \le \frac{1}{m} \sum_{j=1}^m f(\xi^j)^2 \le (1+\ep)\|f\|_2^2.
$$
\end{Theorem}

We note that Theorem \ref{IT2} is more general and slightly stronger than Theorem \ref{Rud}.   Theorem \ref{IT2} provides the Marcinkiewicz-type discretization theorem for a general domain $\Omega$ instead of a discrete set $\Omega_M$. Also, in Theorem \ref{IT2} we have an extra factor $\log N$ instead of $\log\frac{Nt^2}{\ep^2}$  in (\ref{5.1a}). A typical necessary condition for the Marcinkiewicz-type discretization theorem to hold for an $N$-dimensional subspace $X_N$ is $m\ge N$. For instance, such a necessary condition holds when $X_N$ is an $N$-dimensional subspace of continuous functions
with $\Omega =[0,1]$ and the Lebesgue measure on it. 
Both Theorem \ref{Rud} and Theorem \ref{IT2} provide sufficient conditions on $m$ (the upper bound) for existence of a good set of cardinality $m$ for sampling discretization. These sufficient conditions are close to the necessary condition, which is $m\ge N$, but still have an extra $\log N$ factor in the bound for $m$. The main goal of this paper is to prove a sufficient condition on $m$ without an extra $\log N$ factor in the upper bound, which guarantees the Marcinkiewicz-type discretization theorem in $L_2$. This is done in Theorem \ref{IT4}. The first result in that direction was obtained under a condition stronger than Condition E. 

\begin{Theorem}[{\cite[Theorem 4.7]{VT158}}]\label{IT3} Let  $\Omega_M=\{x^j\}_{j=1}^M$ be a discrete set with  the probability measure $\mu_M(x^j)=1/M$, $j=1,\dots,M$. Assume that 
$\{u_i(x)\}_{i=1}^N$ is an orthonormal on $\Omega_M$ system (real or complex). Assume in addition that this system has the following property: for all $j=1,\dots, M$ we have
\be\label{4.19}
\sum_{i=1}^N |u_i(x^j)|^2 = N.
\ee
Then there is an absolute constant $C_1$ such that there exists a subset $J\subset \{1,2,\dots,M\}$ with the property:  $m:=|J| \le C_1 N$ and
 for any $f=\sum_{i=1}^N c_iu_i$ we have  
$$
C_2 \|f\|_2^2 \le \frac{1}{m}\sum_{j\in J} |f(x^j)|^2 \le C_3 \|f\|_2^2, 
$$
where $C_2$ and $C_3$ are absolute positive constants. 
\end{Theorem}

 In Theorem \ref{BT1} we strengthen Theorem \ref{IT3} by showing that the same is true under the weaker Condition E instead of \eqref{4.19}.

 \subsection{Historical comments on weighted discretization}\label{Ic}
 
In the case of weighted discretization, namely, when instead of $\frac{1}{m}\sum_{j=1}^m |f(\xi^j)|^2$ we use the weighted sum $\sum_{j=1}^m\lambda_j |f(\xi^j)|^2$, 
the problem of discretization is solved in the sense of order  in the case of real subspaces $X_N$. It is pointed out in \cite{VT159} that the paper by J. Batson, D.A. Spielman, and N. Srivastava \cite{BSS}  basically solves the discretization problem with weights.  We present an explicit formulation of this important result in our notation.

\begin{Theorem}[{\cite[Theorem 3.1]{BSS}}]\label{BSS} Let  $\Omega_M=\{x^j\}_{j=1}^M$ be a discrete set with the probability measure $\mu_M(x^j)=1/M$, $j=1,\dots,M$, and let $X_N$ be an $N$-dimensional subspace of real functions defined on $\Omega_M$.
Then for any
number $b>1$ there  exists a set of weights $\lambda_j\ge 0$ such that $|\{j: \lambda_j\neq 0\}| \le \lceil bN \rceil$ so that for any $f\in X_N $ we have
\begin{equation*}
\|f\|_2^2 \le \sum_{j=1}^M \lambda_jf(x^j)^2 \le \frac{b+1+2\sqrt{b}}{b+1-2\sqrt{b}}\|f\|_2^2.
\end{equation*}
\end{Theorem}
As observed  in \cite [Theorem 2.13]{DPTT},  this last theorem with  a general probability space $(\Omega,  \mu)$ in place of the discrete space  $(\Omega_M, \mu_M)$ remains true (with other constant in the right hand side) if $X_N\subset L_4(\Omega,\mu)$.  It was proved in \cite{DPSTT2}  that the additional assumption $X_N\subset L_4(\Omega,\mu)$ can be dropped as well.
\begin{Theorem}[{\cite[Theorem 6.3]{DPSTT2}}]\label{DPSTT} If $X_N$ is an $N$-dimensional subspace of the real $L_2(\Omega,\mu)$, then for any $b\in (1,2]$, there exist a set of $m\leq    \lceil bN \rceil$ points $\xi^1,\ldots, \xi^m\in\Omega$ and a set of nonnegative  weights $\lambda_j$, $j=1,\ldots, m$,  such that
\[ \|f\|_2^2\leq  \sum_{j=1}^m \lambda_j f(\xi^j)^2 \leq  \frac C{(b-1)^2}  \|f\|_2^2,\  \ \forall f\in X_N,\]
where $C>1$ is an absolute constant.
\end{Theorem}

  In this paper we obtain analogs of Theorems \ref{BSS} and \ref{DPSTT} in the case of complex subspaces $X_N$ (see Theorems \ref{BT2}, \ref{BT3}, and Remark \ref{BRcr} in Section \ref{B}). Moreover, we provide two different proofs of Theorem \ref{ITw} -- one is based on results 
  from \cite{BSS} and the other is based on results from \cite{MSS}.
  
We note that there are related results on the Banach--Mazur distance
between two finite dimensional spaces of the same dimension (see, for instance, \cite{BLM}, \cite{Sche3}, \cite{JS}).

\section{Main lemma}
\label{A}

Results of this section are based on the following result by A. Marcus, D.A. Spielman and
N. Srivastava.

\begin{Theorem}[{\cite[Corollary 1.5 with $r=2$]{MSS}}]\label{MSS} Let a system of vectors $\bv_1,\dots,\bv_M$ from $\bbC^N$ have the following properties: for all $\bw\in \bbC^N$
\begin{equation}\label{A5}
\sum_{j=1}^M |\<\bw,\bv_j\>|^2 = \|\bw\|_2^2
\end{equation}
and for some $\epsilon>0$
\begin{equation*}
\|\bv_j\|_2^2 \le \epsilon,\qquad j=1,\dots,M.
\end{equation*}
Then there is a partition of $ \{1,2,\dots,M\}$ into two sets $S_1$ and $S_2$ such that for all $\bw\in\bbC^N$ and for each $i=1,2$
\begin{equation*}
  \sum_{j\in S_i} |\<\bw,\bv_j\>|^2 \le \frac{(1+\sqrt{2\epsilon})^2}{2}\|\bw\|_2^2.
\end{equation*}
\end{Theorem} 

The following Lemma \ref{NOUL2} was derived  from Theorem \ref{MSS} in \cite{NOU}   (also see  \cite[Lemma 10.22, p.105]{OU}). 

\begin{Lemma}[{\cite[Lemma 2]{NOU}}]\label{NOUL2} Let a system of vectors $\bv_1,\dots,\bv_M$ from $\bbC^N$ satisfy \eqref{A5} for all $\bw\in \bbC^N$
and
\begin{equation*}
\|\bv_j\|_2^2 = N/M,\qquad j=1,\dots,M.
\end{equation*}
Then there is a subset $J\subset \{1,2,\dots,M\}$ such that for all $\bw\in\bbC^N$
\begin{equation*}
c_0 \|\bw\|_2^2 \le \frac{M}{N} \sum_{j\in J} |\<\bw,\bv_j\>|^2 \le C_0\|\bw\|_2^2,
\end{equation*}
where $c_0$ and $C_0$ are some absolute positive constants. 
\end{Lemma} 

Lemma \ref{NOUL2} does not   control the cardinality of the set $J$, which we need for applications in discretization.  The following simple remark was made in \cite{VT158}. 

\begin{Remark}[{\cite{VT158}}]\label{R4.1} For the cardinality of the subset $J$ from Lemma \ref{NOUL2} we have
$$
c_0 N \le |J| \le C_0 N.
$$
\end{Remark}

 The following lemma is the main lemma for the proof of Theorem \ref{IT4}.  

\begin{Lemma}[{Main lemma}]\label{Lim} Let a system of vectors $\bv_1,\dots,\bv_M$ from $\bbC^N$ satisfy \eqref{A5} for all $\bw\in \bbC^N$
and
\be\label{A6}
\|\bv_j\|_2^2 \le \theta N/M, \qquad  \theta \le M/N,\qquad j=1,\dots,M.
\ee
Then there is a subset $J\subset \{1,2,\dots,M\}$ such that for all $\bw\in\bbC^N$
\be\label{A7}
c_0\theta \|\bw\|_2^2 \le \frac{M}{N} \sum_{j\in J} |\<\bw,\bv_j\>|^2 \le C_0\theta\|\bw\|_2^2,\quad |J|\le C_1\theta N,
\ee
where $c_0$, $C_0$, and $C_1$ are some absolute positive constants. 
\end{Lemma} 

\begin{Remark}\label{LimR} The proof of Lemma \ref{Lim}  gives a slightly stronger result than Lemma \ref{Lim} -- the tight frame condition (\ref{A5}) can be replaced by a frame condition
$$
A\|\bw\|_2^2 \le\sum_{j=1}^M |\<\bw,\bv_j\>|^2 \le B \|\bw\|_2^2,\quad 0<A\le B<\infty.
$$
\end{Remark}
This stronger version of Lemma \ref{Lim} was used in the followup paper \cite{NSU} for sampling recovery.

To derive a complex case analog of Theorem \ref{BSS} (i.e., Theorem \ref{BT2}) from results in \cite{MSS}   we need  to prove the following corollary.

\begin{Corollary}\label{weighted} Let a system of vectors $\bv_1,\dots,\bv_M$ from $\bbC^N$ satisfy \eqref{A5} for all $\bw\in \bbC^N$.
Then there exists a set of weights $\lambda_j\ge 0$, $j=1,\dots, M$, such that $|\{j: \lambda_j\neq 0\}| \le  2C_1N $ and for all $\bw\in\bbC^N$ we have 
\begin{equation*}
c_0\|\bw\|_2^2 \le \sum_{j=1}^M \lambda_j|\<\bw, \bv_j\>|^2 \le C_0\|\bw\|_2^2.
\end{equation*}
 where $c_0$, $C_0$, and $C_1$ are absolute positive constants from Lemma \ref{Lim}.
\end{Corollary} 
\begin{proof}
Without loss of generality we assume that $\|\bv_1\|_2=\min\limits_{j=1,\dots, M}\|\bv_j\|_2$. Let $n_1,\dots, n_M$ be natural numbers such that for every $j$, $1\leq j\leq M$,
\be\label{v_j_norm}
\|\bv_1\|_2^2\leq \frac{\|\bv_j\|_2^2}{n_j}<2\|\bv_1\|_2^2.
\ee
Denote  
\be\label{M'}
M'=\sum\limits_{j=1}^{M} n_j.
\ee
 We build a system $V$ of vectors $\bv_1',\dots, \bv_{M'}'$ from $\bbC^N$ in the following way: for every $j$, $1\leq j\leq M$, we include in $V$ $n_j$ copies of the vector $\bv_j/\sqrt{n_j}$.  
Let us check that $V$ satisfies \eqref{A5} and \eqref{A6} with $\theta=2$. By construction and by our assumption that the system of vectors  $\bv_1,\dots,\bv_M$ satisfies \eqref{A5}, we have
\be\label{A5'}
\sum_{j=1}^{M'} |\<\bw,\bv_j'\>|^2 =\sum_{j=1}^M n_j|\<\bw,\bv_j/\sqrt{n_j}\>|^2 = \|\bw\|_2^2.
\ee
 By construction of the system $V$   we obtain from \eqref{v_j_norm} and \eqref{M'} that 
 \be\label{2_4_2_5}
\|\bv_1\|_2^2 M'\leq \sum_{j=1}^{M} n_j\frac{\|\bv_j\|_2^2}{n_j}=\sum_{j=1}^{M'} \|\bv_j'\|_2^2.
\ee 
Let $e_i$, $i=1,\dots, N$, be the canonical basis of $\bbC^N$. Then from \eqref{A5'} we obtain
\be\label{eqN}
\sum_{j=1}^{M'} \|\bv_j'\|_2^2=
\sum_{j=1}^{M'} \sum_{i=1}^{N}|\< e_i, \bv_j'\>|^2= \sum_{i=1}^{N}\sum_{j=1}^{M'}|\< e_i, \bv_j'\>|^2=\sum_{i=1}^{N}\|e_i\|^2=N.
\ee
Thus, from \eqref{2_4_2_5} and \eqref{eqN} we have $\|\bv_1\|_2^2\leq N/M'$.
 By construction for each $j=1,\dots, M'$, there is a number $k(j)\in\{1,\dots, M\}$ such that 
 $\bv_j'=\bv_{k(j)}/\sqrt{n_{k(j)}}$. Therefore, by \eqref{v_j_norm} we get
$$
\|\bv_j'\|_2^2=\frac{\|\bv_{k(j)}\|_2^2}{n_{k(j)}}<2\|\bv_1\|_2^2\leq 2\frac{N}{M'}, \qquad  j= 1,\dots, M'.
$$
The above inequality implies that the system $V$ satisfies condition (\ref{A6}) and 
equality (\ref{A5'}) implies condition (\ref{A5}). 
We apply Lemma \ref{Lim} to the system $V$ and obtain a subset 
$J\subset \{1,\dots, M'\}$ with $|J|\le 2C_1 N$ such that for all $\bw \in \bbC^N$
$$
c_0 \|\bw\|_2^2 \le \frac{M'}{2N} \sum_{j\in J} |\<\bw,\bv_j'\>|^2 \le C_0\|\bw\|_2^2.
$$
It is clear that 
$$
\frac{M'}{2N} \sum_{j\in J} |\<\bw,\bv_j'\>|^2=
\sum_{j=1}^M \lambda_j|\<\bw, \bv_j\>|^2 
$$
for some nonnegative $\lambda_j, j=1,\dots, M$, so that $|\{j: \lambda_j\neq 0\}| \le  2C_1N $.
\end{proof}

Note that condition (\ref{A5}) implies that $M\ge N$. Lemma \ref{Lim} in some sense improves the celebrated result of M. Rudelson \cite{Rud} where a result similar to Lemma \ref{Lim} was proved with  $|J| \le C_1(t) N \log N$ and with bounds depending on $\epsilon$ (see Theorem \ref{Rud} in Introduction). Proof of Lemma \ref{Lim} uses the iteration method suggested by 
A. Lunin \cite{Lun}.  We also refer the reader to the papers \cite{Ka}, \cite{KL}, \cite{L20} for a discussion of 
recent outstanding progress in the area of submatrices of orthogonal matrices. 
\begin{proof}[{Proof of Lemma \ref{Lim}}]
We use the following known results (for Proposition \ref{AP1} see Corollary B from \cite{NOU}, Corollary 10.19 from \cite{OU}, p.104, or \cite{HO},  and for Lemma \ref{AL1} see Lemma 1 in \cite{NOU} or Lemma 10.20 in \cite{OU}, p.104).

\begin{Proposition}[{\cite[Corollary B]{NOU}}]\label{AP1}
Let $\bv_1,\dots,\bv_M\in\bbC^N$ and $\delta>0$ be such that $\|\bv_j\|_2^2\le\delta$ for all $j=1,\dots,M$. If	
\begin{equation*}
\alpha\|\bw\|_2^2\le\sum\limits_{j=1}^M|\<\bw, \bv_j\>|^2\le \beta\|\bw\|_2^2, \qquad \forall\bw\in\bbC^N, 
\end{equation*}
with some numbers $\beta\ge\alpha>\delta$, then there exists a partition of $\{1,\dots, M\}$ into $S_1$ and $S_2$ such that for each $i=1,2$:
\begin{equation*}
\frac{1-5\sqrt{\delta/\alpha}}{2}\alpha\|\bw\|_2^2\le\sum\limits_{j\in S_i}|\<\bw, \bv_j\>|^2\le \frac{1+5\sqrt{\delta/\alpha}}{2}\beta\|\bw\|_2^2, \quad\forall\bw\in\bbC^N.
\end{equation*}
\end{Proposition}

\begin{Lemma}[{\cite[Lemma 1]{NOU}}]\label{AL1}
	Let $0<\delta<1/100$, and let $\alpha_j,\beta_j, j=0,1,\dots$, be defined inductively
	$$
\alpha_0=\beta_0=1, \quad \alpha_{j+1}:=\alpha_j\frac{1-5\sqrt{\delta/\alpha_j}}{2}, \quad \beta_{j+1}:=\beta_j\frac{1+5\sqrt{\delta/\alpha_j}}{2}.
	$$
Then there exist a positive absolute constant $C$ and a number $L\in\mathbb{N}$ such that 
$$
\alpha_j\geq 100\delta,\quad j\le L, \quad 25\delta\le \alpha_{L+1}<100\delta, \quad \beta_{L+1}<C\alpha_{L+1}.
$$
\end{Lemma}

If $\delta:=\theta N/M\ge 1/100$, then
 (\ref{A7}) holds with $J=\{1,2,\dots,M\}$ and $C_1= 1/\delta\le 100$, $c_0=1$, $C_0=100$. Assume $\delta<1/100$. Let $\alpha_j, \beta_j$ be as defined in Lemma \ref{AL1}; then the vectors $\bv_1,\dots, \bv_M$ satisfy the assumptions of Proposition \ref{AP1} with  $\alpha=\beta=1$.  We apply Proposition \ref{AP1} and choose a subset of the obtained partition with a smaller cardinality. We obtain a set $J_1\subset\{1,2,\dots, M\}$ with $|J_1|\leq M/2$ such that for all $ \bw\in\bbC^N$
\begin{equation*}
  \alpha_1\|\bw\|_2^2\le \sum\limits_{i\in J_1}|\<\bw,\bv_i\>|^2\le \beta_1\|\bw\|_2^2.
\end{equation*}
Since $\alpha_1>25\delta$ we can apply Proposition \ref{AP1} again and obtain $J_2\subset J_1$ with $|J_2|\le M/2^2$, for which we have two-sided inequalities  with $\alpha_2>0$ and $\beta_2$. Let $L$ be the number from Lemma \ref{AL1}. We iteratively apply 
Proposition \ref{AP1} (choosing at each step the subset  $S_i$  with the smallest cardinality) and find $J_1\supset J_2\supset\dots \supset J_{L+1}$ with the property
$$
\frac{1-5\sqrt{\delta/\alpha_L}}{2}\alpha_L\|\bw\|_2^2\le\sum\limits_{j\in J_{L+1}}|\<\bw, \bv_j\>|^2\le \frac{1+5\sqrt{\delta/\alpha_L}}{2}\beta_L\|\bw\|_2^2, \quad\forall\bw\in\bbC^N.
$$
By Lemma \ref{AL1} we obtain
$$
\frac{1-5\sqrt{\delta/\alpha_L}}{2}\alpha_L = \alpha_{L+1} \ge 25\delta,
$$
$$
 \frac{1+5\sqrt{\delta/\alpha_L}}{2}\beta_L=\beta_{L+1} \le C\alpha_{L+1} < 100C\delta.
 $$
 Thus, for $J:=J_{L+1}$ we have
\begin{equation*} 
25\theta\frac{N}{M}\|\bw\|_2^2\le \sum\limits_{i\in J}|\<\bw,\bv_i\>|^2\le 100C\theta\frac{N}{M}\|\bw\|_2^2.
\end{equation*}
Note that 
 $2^{-L-1}\le
\beta_{L+1}<100C\delta$, therefore
$|J_{L+1}|\le M/2^{L+1}\le 100CM\delta=100C\theta N$ as required.
\end{proof}

\section{Application to discretization}
\label{B}
	
The following corollary of Lemma \ref{Lim} is a generalization of Theorem $4.7$ from \cite{VT158} (see Theorem \ref{IT3} in Introduction). In \cite{VT158} instead of condition (\ref{4}) a stronger assumption  (\ref{4.19})  was imposed. 

\begin{Theorem}\label{BT1} Let  $\Omega_M=\{x^j\}_{j=1}^M$ be a discrete set with the probability measure $\mu_M(x^j)=1/M$, $j=1,\dots,M$. Assume that 
$\{u_i(x)\}_{i=1}^N$ is an orthonormal on $\Omega_M$ system (real or complex). Assume in addition that this system has the following property: for  all $j=1,\dots, M$  and for some $t>0$ we have
\be\label{4}
\sum_{i=1}^N |u_i(x^j)|^2 \le  Nt^2.
\ee
Then there is an absolute  constant $C_1$ such that there exists a subset $J\subset \{1,2,\dots,M\}$ with the property:  $m:=|J| \le C_1t^2 N$ and
 for any $f=\sum_{i=1}^N c_iu_i$  we have  
\be\label{4'}
C_2 \|f\|_2^2 \le \frac{1}{m}\sum_{j\in J} |f(x^j)|^2 \le C_3t^2 \|f\|_2^2, 
\ee
where $C_2$ and $C_3$ are absolute positive constants. 
\end{Theorem}
\begin{proof} Define the column vectors
\begin{equation*}
\bv_j := M^{-1/2}(u_1(x^j),\dots,u_N(x^j))^T,\qquad j=1,\dots, M.
\end{equation*}
Then our assumption (\ref{4}) implies that the system $\bv_1,\dots,\bv_M$ satisfies (\ref{A6}) with $\theta=t^2$. For any $\bw=(w_1,\dots,w_N)^T\in \bbC^N$ we have 
$$
\sum_{j=1}^M |\<\bw,\bv_j\>|^2 = \frac{1}{M}\sum_{j=1}^M \sum_{i,k=1}^N w_i {\bar w}_k {\bar u}_i(x^j) u_k(x^j) = \sum_{i=1}^N |w_i|^2
$$
by the orthonormality assumption. This implies that the system $\bv_1,\dots,\bv_M$ satisfies (\ref{A5}). 

Note that the necessary condition for (\ref{4'}) to hold is $m\ge N$. Applying Lemma \ref{Lim} we complete the proof of  Theorem \ref{BT1}.   
\end{proof}

The following Theorem \ref{BT2}, which is a complex analog of Theorem \ref{BSS}, can  be derived from Corollary \ref{weighted} in the same way as we have derived Theorem \ref{BT1} from Lemma \ref{Lim} above.

\begin{Theorem}\label{BT2}
Let  $\Omega_M=\{x^j\}_{j=1}^M$ be a discrete set with the probability measure $\mu_M(x^j)=1/M$, $j=1,\dots,M$. Assume that 
$\{u_i(x)\}_{i=1}^N$ is an orthonormal on $\Omega_M$ system (real or complex). 
Then there is an absolute  constant $C_1$ such that there exists a set of weights $\lambda_j\geq 0$, $j=1,\dots, M$, with the property:  $m:=|\{j:\lambda_j\neq 0\}| \le C_1N$ and
    for any $f=\sum_{i=1}^N c_iu_i$  we have  
\begin{equation*}
c_0 \|f\|_2^2 \le \sum_{j=1}^M \lambda_j|f(x^j)|^2 \le C_0 \|f\|_2^2, 
\end{equation*}
where $c_0$ and $C_0$ are from Lemma \ref{Lim}. 
\end{Theorem}

Further, using Theorem \ref{BT2} and repeating the argument in the proof of Theorem 6.3 from \cite{DPSTT2} (with natural modifications from the real case to the complex case), which was used to derive Theorem \ref{DPSTT} from Theorems \ref{BSS} and \ref{IT2}, we obtain the complex analog of Theorem \ref{DPSTT} -- Theorem \ref{ITw}, which we formulate below as Theorem \ref{BT3} for the reader's convenience. Note that the complex version of Theorem 
\ref{IT2} can be proved in the same way as Theorem \ref{IT2} was proved in \cite{VT159}.

\begin{Theorem}\label{BT3} If $X_N$ is an $N$-dimensional subspace of the complex $L_2(\Omega,\mu)$, then there exist three absolute positive constants $C_1'$, $c_0'$, $C_0'$,  a set of $m\leq   C_1'N$ points $\xi^1,\ldots, \xi^m\in\Omega$, and a set of nonnegative  weights $\lambda_j$, $j=1,\ldots, m$,  such that
\[ c_0'\|f\|_2^2\leq  \sum_{j=1}^m \lambda_j |f(\xi^j)|^2 \leq  C_0' \|f\|_2^2,\  \ \forall f\in X_N.\]
\end{Theorem}

\begin{Remark}\label{BR1}  A combination of the proof of Theorem 6.3 from \cite{DPSTT2} with 
Theorem \ref{BT4} (in the proof of Theorem 6.3 from \cite{DPSTT2} we use Theorem \ref{BT4} instead of Theorem 1.2 (Theorem \ref{IT2} above)) gives Theorem \ref{BT3} with $C_1'=C_1'$, $c_0'=C_2'$, and $C_0'=C_3'$, where $C_i'$, $i=1,2,3$, are from Theorem \ref{BT4}.  It is another way to prove Theorem \ref{BT3}.
\end{Remark}

It is important to emphasize that in the proofs of Theorems \ref{BT2} and \ref{BT3}, which are complex companions of Theorems \ref{BSS} and \ref{DPSTT}, we did not use Theorems \ref{BSS} and \ref{DPSTT}. Thus, our arguments give other proofs of analogs of Theorems \ref{BSS} and \ref{DPSTT}. Note that constants in Theorems \ref{BT2} and \ref{BT3}  are not as good as constants in Theorems \ref{BSS} and \ref{DPSTT}.

{
\bigskip

{\bf A comment on connection between real and complex weighted discretization.} We show here that good discretization of the $L_2(\Omega,\mu)$-norm of functions from real subspaces of dimension $2N$ implies good discretization of the $L_2(\Omega,\mu)$-norm of functions from complex subspaces of dimension $N$.}
\begin{Definition}\label{BD1} {Let $X_N$ be a subspace of $L_2(\Omega,\mu)$. For $m\in\N$ and positive constants $C_1\le C_2$ we write $X_N \in \cM^w(m,2,C_1,C_2)$ if there exist a set of points $\xi^1,\dots,\xi^m \in \Omega$ and a set of weights $\la_\nu$, $\nu=1,\dots,m$, such that for any $f\in X_N$ we have
\be\label{BA3}
C_1\|f\|_2^2 \le  \sum_{\nu=1}^m \la_\nu |f(\xi^\nu)|^2 \le C_2\|f\|_2^2.
\ee}
\end{Definition} 

\begin{Proposition}\label{BPcr} {Let $X_N=\sp(w_1,\dots,w_N)$ be a subspace of complex $L_2(\Omega,\mu)$. Suppose that $w_j=u_j+iv_j$, where $u_j$, $v_j$ are real functions, $j=1,\dots,N$. 
Denote $Y_{S}:=\sp(u_1,\dots,u_N,v_1,\dots,v_N)$, $S:=\dim Y_S\le 2N$, a real subspace of $L_2(\Omega,\mu)$. 
Then
$$
Y_{S} \in \cM^w(m,2,C_1,C_2)\quad \text{implies} \quad X_N \in \cM^w(m,2,C_1,C_2).
$$
Moreover, for discretization of $X_N$ we can use the same points and weights as for discretization of $Y_{S}$.}
\end{Proposition}
\begin{proof} {Take an $f\in X_N$ and write
$$
f=  f_R+if_I,\quad f_R, f_I \in Y_{S}.
$$
Assume that a set of points $\xi^1,\dots,\xi^m \in \Omega$ and a set of weights $\la_\nu$, $\nu=1,\dots,m$, are such that for any $g\in Y_{S}$ we have
\be\label{BA4}
C_1\|g\|_2^2 \le  \sum_{\nu=1}^m \la_\nu |g(\xi^\nu)|^2 \le C_2\|g\|_2^2.
\ee
Then on one hand
$$
 \sum_{\nu=1}^m \la_\nu |f(\xi^\nu)|^2 =  \sum_{\nu=1}^m \la_\nu (|f_R(\xi^\nu)|^2+|f_I(\xi^\nu)|^2)
 \le C_2(\|f_R\|_2^2 +\|f_I\|_2^2) = C_2\|f\|_2^2.
 $$
 On the other hand
 $$
 \sum_{\nu=1}^m \la_\nu |f(\xi^\nu)|^2 =  \sum_{\nu=1}^m \la_\nu (|f_R(\xi^\nu)|^2+|f_I(\xi^\nu)|^2)
 \ge C_1(\|f_R\|_2^2 +\|f_I\|_2^2) = C_1\|f\|_2^2.
 $$
The above inequalities prove Proposition \ref{BPcr}.}
\end{proof}

\begin{Remark}\label{BRcr} {Proposition \ref{BPcr} and Theorem \ref{DPSTT} imply that Theorem \ref{ITw} holds with $m\leq    \lceil 2bN \rceil$, $b\in (1,2]$, and $c_0'=1$, $C_0'= C(b-1)^{-2}$, where $C$ is an absolute constant from Theorem \ref{DPSTT}.}
\end{Remark}

 {\bf A remark on sampling recovery.} We mentioned in the Introduction that inequality (\ref{osrin}) 
was proved in \cite{VT183} with the help of Theorem \ref{ITw}. We now point out that if, in the proof of (\ref{osrin}) (from \cite{VT183}), we replace
Theorem \ref{ITw} with  either Theorem \ref{DPSTT}  (real case) or Remark \ref{BRcr} (complex case), 
then we obtain the following version of  (\ref{osrin}).
\begin{Theorem}\label{SRb}  For any $b\in(1,2]$ there exists a positive constant $B=B(b)$ such that for any   compact subset $\Omega$  of $\R^d$, any probability measure $\mu$ on it, and any compact subset $\bF$ of $\C(\Omega)$ we have in the real case
$$
\ro_{\lceil b(n+1) \rceil}(\bF,L_2(\Omega,\mu)) \le Bd_n(\bF,L_\infty)
$$
and in the complex case
$$
\ro_{\lceil b(2n+1) \rceil}(\bF,L_2(\Omega,\mu)) \le Bd_n(\bF,L_\infty).
$$
\end{Theorem}

 Let $\Omega$ be a  nonempty compact set in $\R^d$ and let  $X_N$ be an $N$-dimensional subspace of real (or complex) space of continuous functions $\mathcal C(\Omega)$. Let $\mu$ be a probability measure on $\Omega$ and let $\{u_i(x)\}_{i=1}^N$ be an orthonormal basis for $X_N$.

{\bf Nikol'skii inequality.} We say that $X_N$ satisfies the Nikol'skii inequality for the pair $(2,\infty)$ if there exists a constant  $t>0$ such that
\begin{equation}\label{6}
\|f\|_\infty \leq t N^{\frac12}\|f\|_2,\   \ \forall f\in X_N.
\end{equation}
We point out that condition \eqref{6} with $X_N= \sp(u_1,\dots,u_N)$, where $\{u_j\}_{j=1}^N$ is an orthonormal system, is equivalent to Condition E. This can be seen from the following simple well-known result, which is a corollary of the Cauchy inequality.		
\begin{Proposition}\label{P1}  Let $X_N$ be an $N$-dimensional subspace of $\mathcal C(\Omega)$.
	Then for any orthonormal basis $\{u_i\}_{i=1}^N$ of $X_N\subset L_2(\Omega,\mu)$ we have that  for $x\in \Omega$
\begin{equation*} 
	\sup_{f\in X_N; f\neq 0}|f(x)|/\|f\|_2 =  \left(\sum_{i=1}^N |u_i(x)|^2\right)^{1/2}  .
\end{equation*}
\end{Proposition}

 The following simple result can be found in \cite{DPTT}. Note that only the real case is discussed in \cite{DPTT}. However, the same argument works for the complex case as well.
 
 \begin{Proposition}[{\cite[Proposition 2.1]{DPTT}}]\label{P2} Let $Y_N:=\sp(u_1(x),\dots,u_N(x))$ with $\{u_i(x)\}_{i=1}^N$ being a real (or complex) orthonormal on $\Omega$ with respect to a probability measure $\mu$ basis for $Y_N$. Assume that $\|u_i\|_4:=\|u_i\|_{L_4(\Omega,\mu)} <\infty$ for all $i=1,\dots,N$.   Then for any $\de>0$ there exists
 a set $\Omega_M=\{x^j\}_{j=1}^M\subset \Omega$ such that for any $f\in Y_N$
 \begin{equation*}
| \|f\|_{L_2(\Omega,\mu)}^2 - \|f\|_{L_2(\Omega_M,\mu_M)}^2| \le \de \|f\|_{L_2(\Omega,\mu)}^2,
 \end{equation*}
 where
 $$
 \|f\|_{L_2(\Omega_M,\mu_M)}^2 := \frac{1}{M}\sum_{j=1}^M |f(x^j)|^2.
 $$
 \end{Proposition}
 The following generalization of Theorem \ref{BT1}, which is equivalent to Theorem \ref{IT4}, is the main result of the paper. 
 \begin{Theorem}\label{BT4} Let  $\Omega\subset \R^d$ be a nonempty compact set with the probability measure $\mu$. Assume that $X_N\subset \mathcal C(\Omega)$ satisfies the Nikol'skii inequality (\ref{6}). 
Then there is an absolute  constant $C_1'$ such that there exists a set $\{\xi^j\}_{j=1}^m\subset \Omega$ of $m \le C_1' t^2 N$ points with the property:
 for any $f\in X_N $ we have  
\be\label{9}
C_2' \|f\|_2^2 \le \frac{1}{m}\sum_{j=1}^m |f(\xi^j)|^2 \le C_3' t^2\|f\|_2^2, 
\ee
where $C_2'$ and $C_3'$ are absolute positive constants. 
\end{Theorem}
\begin{proof} For a given $\delta \in (0,1)$, taking into account Proposition \ref{P2}, we find a set $\Omega_M=\{x^j\}_{j=1}^M$ such that for any $f\in X_N$
 \be\label{10}
| \|f\|_{L_2(\Omega,\mu)}^2 - \|f\|_{L_2(\Omega_M,\mu_M)}^2| \le \de \|f\|_{L_2(\Omega,\mu)}^2.
\ee
Specify $\delta=1/2$.
Then, clearly, subspace $X_N$ restricted to $\Omega_M$ (denote it by $Y_l$)  satisfies the Nikol'skii inequality (\ref{6}) with $t$ replaced by $2t$. Let $u_1,\dots, u_ l$, $l\le N$, be an orthonormal basis of $Y_l$.  By Proposition \ref{P1} inequality (\ref{6}) is equivalent to (\ref{4}). 
Now applying Theorem \ref{BT1} to $Y_l$ we find a subset $J\subset \{1,2,\dots,M\}$ with the property:  $m:=|J| \le C_1(2t)^2 N$ and
 for any $f\in X_N$ we have  
$$
C_2 \|f\|_{L_2(\Omega_M,\mu_M)}^2 \le \frac{1}{m}\sum_{j\in J} |f(x^j)|^2 \le C_3 t^2\|f\|_{L_2(\Omega_M,\mu_M)}^2, 
$$
where $C_2$ and $C_3$ are absolute positive constants from Theorem \ref{BT1}. From here and (\ref{10}) with $\delta=1/2$ we obtain (\ref{9}).
\end{proof}

\begin{Remark}\label{BR2} In Theorem \ref{BT4} we assume that $X_N\subset \mathcal C(\Omega)$. It is done for convenience. The statement of Theorem \ref{BT4} holds if instead of 
continuity assumption we require that $X_N$ is a subspace of the space $\mathcal B(\Omega,\mu)$ of functions, which are bounded and measurable with respect to $\mu$ on $\Omega$, {where $\Omega$ is a nonempty subset of $\R^d$}. 
\end{Remark}

{\bf Acknowledgements.} The authors are grateful to Boris Kashin for very useful comments, in particular, for bringing the paper \cite{NOU} to their attention. 

The work was supported by the Russian Science Foundation (project no. 21-11-00131) at Lomonosov Moscow State University  (Sections 1,2) and by the Russian Federation Government Grant N{\textsuperscript{\underline{o}}}14.W03.31.0031 (Section 3)

\end{document}